\documentclass[a4paper,11pt]{article}
\usepackage[active]{srcltx}
\usepackage{array}
\usepackage{theorem}
\usepackage{amsmath,amscd,amssymb}
\usepackage{latexsym}
\usepackage{epsfig}
\usepackage{xypic}
\usepackage{setspace}
\theorembodyfont{\sl}

\newtheorem{lemma}{Lemma}[section]

\newtheorem{proposition}[lemma]{Proposition}
\newtheorem{theorem}[lemma]{Theorem}

\newtheorem{definition}[lemma]{Definition}

\newcommand{\CC}{\mathbb C}

\newcommand{\HH}{\mathbb H}

\newcommand{\PP}{\mathbb P}
\newcommand{\QQ}{\mathbb Q}
\newcommand{\RR}{\mathbb R}

\newcommand{\ZZ}{\mathbb Z}
\newcommand{\cA}{\mathcal A}

\renewcommand{\cD}{\mathcal D}

\newcommand{\cF}{\mathcal F}

\newcommand{\cM}{\mathcal M}

\newcommand{\cS}{\mathcal S}

\newcommand{\emb}{\hookrightarrow}

\renewcommand{\Tilde}{\widetilde}

\renewcommand{\Bar}{\overline}

\newcommand{\SL}{\mathop{\mathrm {SL}}\nolimits}
\newcommand{\SO}{\mathop{\mathrm {SO}}\nolimits}
\newcommand{\Sp}{\mathop{\mathrm {Sp}}\nolimits}

\newcommand{\Orth}{\mathop{\null\mathrm {O}}\nolimits}

\newcommand{\rank}{\mathop{\mathrm {rank}}\nolimits}
\newcommand{\sign}{\mathop{\mathrm {sign}}\nolimits}

\newcommand{\reg}{\mathop{\mathrm {reg}}\nolimits}

\newcommand{\id}{\mathop{\mathrm {id}}\nolimits}

\newcommand{\Pic}{\mathop{\null\mathrm {Pic}}\nolimits}

\newcommand{\supp}{\mathop{\null\mathrm {Supp}}\nolimits}
\newcommand{\Bdiv}{\mathop{\null\mathrm {Bdiv}}\nolimits}

\newcommand{\latt}[1]{{\langle{#1}\rangle}}
\renewcommand{\div}{\mathop{\mathrm {div}}\nolimits}
\newcommand{\Kthree}{\mathop{\mathrm {K3}}\nolimits}

\newcommand{\qedsymbol}{\mbox{$\Box$}}
\newcommand{\qed}{\unskip\nobreak\hfil\penalty50\hskip1em\hbox{}\nobreak
\hfill\qedsymbol\parfillskip=0pt\finalhyphendemerits=0}
\newenvironment{proof}{\begin{ProofwCaption}{Proof}}{\end{ProofwCaption}}
\newenvironment{ProofwCaption}[1]
 {\addvspace\theorempreskipamount \noindent{\it #1.}\rm}
 {\qed \par \addvspace\theorempostskipamount}

\begin{document}

\title{Uniruledness of orthogonal modular varieties}
\author{V.~Gritsenko and K.~Hulek}
\maketitle

\begin{abstract}
A strongly reflective modular form with respect to  
an orthogonal group of signature $(2,n)$ determines a Lorentzian Kac--Moody algebra.
We find a new geometric application of such modular  forms:
we prove that if the weight is larger 
than $n$ then the corresponding modular variety 
is uniruled. We also construct new reflective modular forms
and thus provide new examples of uniruled  moduli spaces of lattice polarised
$\Kthree$ surfaces. Finally we  prove  that the moduli space
of Kummer surfaces associated to $(1,21)$-polarised abelian surfaces
is uniruled.
\end{abstract}

\section{Reflective modular forms}

Let $L$ be an even integral   lattice with a
quadratic form of signature $(2,n)$ and let
$$
  \cD(L)=\{[Z] \in \PP(L\otimes \CC) \mid
  (Z,Z)=0,\ (Z,\Bar Z)>0\}^+
$$
be the associated $n$-dimensional bounded symmetric  Hermitian domain  of type $IV$
(here $+$ denotes one of its two connected components).
We denote by $\Orth^+(L)$ the index $2$ subgroup of the integral orthogonal group $\Orth(L)$
preserving $\cD(L)$.
For any $v\in L\otimes \QQ$ such that $v^2=(v,v)<0$ we define  the {\it rational quadratic divisor}
$$
\cD_v=\cD_v(L)=\{[Z] \in \cD(L)\mid (Z,v)=0\}\cong \cD(v^\perp_L)
$$
where $v^\perp_L$ is an even integral  lattice of signature $(2,n-1)$.
If $\Gamma<\Orth^+(L)$ is of finite index we
define  the corresponding {\it modular variety}
$$
\cF_L(\Gamma)=\Gamma\backslash \cD(L),
$$
which is a quasi-projective variety of dimension $n$.
The most important  subgroups of $\Orth^+(L)$ are
the stable orthogonal groups
\begin{equation*}
\Tilde{\Orth}^+(L) =\{g\in \Orth^+(L)\mid g|_{L^\vee/L}=\id\},\quad
\Tilde{\SO}^+(L)=\SO(L)\cap \Tilde{\Orth}^+(L)
\end{equation*}
where $L^\vee$ is the dual lattice of $L$.
Modular varieties of orthogonal type appear in algebraic geometry.
A prime example are moduli spaces of polarised abelian surfaces and $\Kthree$ surfaces  
or, more generally,  the moduli spaces    of polarised  holomorphic symplectic 
varieties (see \cite{GHS2}--\cite{GHS3}).

Let $k>0$ and $\chi : \Gamma\to \CC^*$ be a character or multiplier system
(of finite order) of $\Gamma$.
By $M_k(\Gamma,\chi)$ we denote the space of modular forms of weight $k$
and character $\chi$ with respect to $\Gamma$.

\begin{definition}\label{def-reflf}
A modular form $F\in M_k(\Gamma,\chi)$ is called {\bf reflective} if
\begin{equation}\label{eq-reflf}
\supp (\div F) \subset \bigcup_{\substack{ r\in L/\pm 1 \vspace{1\jot}\\
r\  {\rm is\   primitive}\vspace{1\jot}\\
\sigma_r\in \Gamma\text{ or }-\sigma_r\in \Gamma}} \cD_r(L)
\end{equation}
where $\sigma_r:l\to l-\frac{2(r,l)}{(r,r)}r$ is the reflection with respect
to $r$. 
We call $F$ {\bf strongly reflective} if the multiplicity of any irreducible component
of $\div F$ is equal to one.
\end{definition}

This definition  is motivated  by the following result
proved  in \cite[Corollary 2.13]{GHS1}.
\begin{proposition}\label{pr-brdiv}
Let $\sign(L)=(2,n)$ and $n \geq 3$. 
The union of the rational quadratic  divisors in (\ref{eq-reflf})
is equal to the ramification divisor $\Bdiv(\pi_\Gamma)$ of the modular projection
$$
\pi_\Gamma: \cD(L)\to \Gamma\backslash \cD(L).
$$
\end{proposition}

The most famous example of a strongly reflective modular form is the  Borcherds
form $\Phi_{12} \in M_{12}(\Orth^+(II_{2,26}), \det)$ defined in \cite{B}.
It is known  that 
$$
\div \Phi_{12}=\bigcup_{\substack{r\in R_{-2}(II_{2,26})/\pm 1 }} 
\cD_r(II_{2,26})
$$
where $R_{-2}(II_{2,26})$ denotes the set of $-2$-vectors in the even unimodular 
lattice $II_{2,26}$.

Strongly reflective modular forms are very rare. They 
determine Lorentz\-ian Kac--Moody algebras (see \cite{B}, \cite{GN-P2}). 
The  following theorem proved in 2010 shows that the  existence of a strongly
reflective modular form of large weight $k\ge n$ implies that 
the corresponding modular variety has special geometric properties.

\begin{theorem}(see \cite{G-R})\label{thm-kdim}
Let $\sign(L)=(2,n)$ and $n\ge 3$.
Let $F_k\in M_k(\Gamma, \chi)$ be  a strongly reflective modular form
of weight $k$ and character (of finite order) $\chi$ where $\Gamma<\Orth^+(L)$ is of finite index.
By $\kappa(X)$ we denote the Kodaira dimension of $X$.
Then
$$
\kappa(\Gamma\backslash \cD(L))=-\infty
$$
if $k>n$, or $k=n$ and $F_k$ is not a cusp form.
If $k=n$ and $F_k$ is a cusp form
then
$$
\kappa(\Gamma_\chi\backslash \cD(L))=0,
$$
where $\Gamma_\chi=\ker(\chi\cdot \det)$ is a subgroup of $\Gamma$.
\end{theorem}

Below we prove a stronger theorem which allows us to conclude  that
the variety   $\Gamma\backslash \cD(L)$ is {\it uniruled}.
As an application, using reflective modular forms,
we give new examples of uniruled moduli spaces in \S 3.

\section{A sufficient criterion for unirulednesss of orthogonal modular varieties}

Recall that a variety $X$ is called {\it uniruled} if there exists a dominant rational 
map $Y \times \PP^1 \dasharrow X$ where $Y$ is a variety with $\dim Y = \dim X -1$. 
If $Y$ is uni\-ruled, then $\kappa(Y)= - \infty$. A well known conjecture says that 
the converse also holds, but
this is not known with the exception of dimension $3$ (where it follows from \cite{Mi}).
Using results of Boucksom, Demailly, Paun and Peternell \cite{BDPP} the conjecture would follow from 
the abundance conjecture.
We shall use the numerical criterion for uniruledness due to Miyaoka and Mori \cite{MM} to
formulate a criterion which allows us to prove uniruledness of orthogonal modular varieties in many cases.

\begin{theorem}\label{Th-uni}
Let $\cD=\cD (L)$ be a connected component of the type $IV$ domain associated to a lattice $L$ of
signature $(2,n)$ with $n \geq 3$ and let $\Gamma \subset \Orth^+(L)$ be an arithmetic group. Let 
$\widetilde{B}= \sum_{r} \cD_{r}$ in $\cD$ be the divisorial part of the ramification
locus of the quotient map $ \cD \to \Gamma \backslash \cD$
(see (\ref{eq-reflf}) and Proposition \ref{pr-brdiv}). Assume that a modular form $F_k$ with respect to
$\Gamma$ of weight $k$ with a (finite order) character exists, such that 
$$
\{F_k=0\} = \sum_{r} m_{r} \cD_{r} 
$$
where the $m_{r}$ are non-negative integers.
Let $m = \max\{m_{r}\}$ (which must be $> 0$ by Koecher's principle).
If $k >  m\cdot n$,
then $\Gamma' \backslash \cD$ is uniruled (and thus in particular 
has Kodaira dimension $-\infty$) for every arithmetic group $\Gamma'$ contating $\Gamma$.
\end{theorem}
\begin{proof}
It is clearly enough to prove the result for $\Gamma$ since every finite quotient of a uniruled variety 
is uniruled itself.

We first recall that by \cite[Theorem 2.12]{GHS1} every
quasi-reflection in $h \in \Gamma$ has the property that
$h^2 = \pm \id$ and thus $h$ acts as a reflection on $\cD$. 
We choose a toroidal compactification $X'$ of the quotient 
$X=\Gamma \backslash \cD$ for which we can 
assume that the  boundary contains no ramification divisor. Such a compactification exists by
\cite[Corollary 2.22]{GHS1} and the proof of \cite[Corollary 2.29]{GHS1}.  Let $L$ be the 
($\QQ$-)line bundle of modular forms of weight $1$.
We denote the branch locus in $X'$ by $B=\sum_{r}\cD'_{r}$ (here we use, by abuse 
of notation, the same index set for the components as we do for the ramification locus).
Then over the regular part $X'_{\reg}$ of $X'$ we have
\begin{equation}\label{equ:canonial}
K_{X'_{\reg}}=nL- \frac{1}{2}B - D
\end{equation}
where $D=\sum_{\alpha}D_{\alpha}$ is the boundary. 

The assumption about the vanishing locus of the form $F_k$ implies
\begin{equation*}\label{equ:divFk}
kL= \frac12 \,\sum_{r}m_{r}\cD'_{r} + \sum_{\alpha} \delta_{\alpha}D_{\alpha}, 
\quad \delta_{\alpha} \geq 0.
\end{equation*}
Note that the factor $1/2$ in front of the term involving the $\cD'_r$
comes from the fact that 
the map $\cD(L) \to X$ is branched of order $2$ along $\widetilde{B}$.
We rewrite this as 
\begin{equation*}\label{equ:divFk2}
kL= \frac12 \big(m B + \sum_{r}(m_{r} - m)\cD'_{r}\big) + 
\sum_{\alpha} \delta_{\alpha}D_{\alpha}
\end{equation*}
and use this to eliminate $\frac 12 B$ from formula (\ref{equ:canonial}). The result is
\begin{equation*}\label{equ:minuscanonical}
-K_{X'_{\reg}}=(\frac{k}{m} - n)L + 
\sum_{r}\frac{m - m_{r}}{2m}\cD'_{r} +
\sum_{\alpha} \frac{m - \delta_{\alpha}}{m}D_{\alpha}.
\end{equation*}
Next we choose a resolution $\widetilde X \to X'$. For the canonical bundle on $\widetilde X$ we obtain
\begin{equation*}\label{equ:minus_canonical_regular}
-K_{\widetilde X}=(\frac{k}{m} - n)L + 
\sum_{r}\frac{m - m_{r}}{2m}\cD'_{r} +
\sum_{\alpha} \frac{m - \delta_{\alpha}}{m}D_{\alpha} +
\sum_{\beta} \varepsilon_{\beta} E_{\beta}
\end{equation*}
where the $E_{\beta}$ are the exceptional divisors and where we have used the notation $\cD'_{r}$
and $D_{\alpha}$ also for the strict transform of the corresponding divisors on $X'$.

We recall the following criterion of Mori and Miyaoka\cite[Theorem 1]{MM}: 
assume that a smooth projective variety $Z$ 
contains an open subest $U$ such that through every point $x \in U$ there is a curve $C$ with 
$K_Z.C <0$. Then $Z$ is uniruled. We want to apply this to $\widetilde X$ where we choose $U=X_{\reg}$.
Recall that a high multiple $L^{\otimes n_0}$ of $L$ defines a map 
$\varphi_{L^{\otimes n_0}}: \widetilde X \to \PP^N $ whose image is 
the Baily-Borel compactification $X^{\operatorname {BB}}$ of $X$. 
The restriction of this map to $U$ is an isomorphism onto
the image and the codimension of $X^{\operatorname {BB}} \setminus U$ in $X^{\operatorname {BB}}$ is at least $2$ since
the boundary of the Baily-Borel compactification is $1$-dimensional and $X$ is normal. 

Let $x \in U$. Intersecting with $n-1$ general hyperplanes through $x$ we obtain a curve $C$ which misses both 
the boundary and the singular locus of $X$. Hence we can also consider $C$ as a curve in $\widetilde X$.
We find that
\begin{equation*}\label{equ:intersectionnumber}
-K_{\widetilde X}.C=\left((\frac{k}{m} - n)L + 
\sum_{r}\frac{m - m_{r}}{2m}\cD'_{r} +
\sum_{\alpha} \frac{m - \delta_{\alpha}}{m}D_{\alpha} + 
\sum_{\beta} \varepsilon_{\beta} E_{\beta}\right).C 
\end{equation*}
Since $C$ does not meet the boundary and the singular locus of $X$ we have 
$D_{\alpha}.C=E_{\beta}.C=0$. 
Moreover, since $m \geq m_r$ and $\cD'_r.(n_0L)^{n-1} >0$ we have 
$\frac{m - m_{r}}{2m}\cD'_{r}.C \geq 0$. Finally, since $k > m\cdot n$ 
and $L.(n_0L)^{n-1}>0$
it follows that $K_{\widetilde X}.C <0$ and
thus we can apply the criterion by Miyaoka and Mori.
\end{proof}

\section{New examples of uniruled moduli spaces}

{\bf 3.1 The moduli space of Kummer surfaces associated to $(1,21)$-polarised
abelian surfaces.}
\smallskip

The moduli space of $(1,t)$-polarised abelian surfaces is 
a Siegel modular $3$-fold ${\cal A}_{t}=\Gamma_{t}\setminus \HH_2$
where $\HH_2$ is the Siegel upper-half plane of genus $2$ and 
$\Gamma_t$  is the corresponding  paramodular group which is isomorphic 
to the integral symplectic group of the symplectic form with elementary divisors
$(1,t)$.
The paramodular group $\Gamma_t$ has the maximal extension $\Gamma_t^*$
in $\Sp_2(\RR)$ of order $2^{\nu(t)}$ where $\nu(t)$ is the number of 
prime divisors of $t$ (see \cite{GH-M}). We proved
in \cite[Theorem 1.5]{GH-M} that the modular variety   
${\cal A}_{t}^*=\Gamma_{t}^*\setminus \HH_2$  
can be considered as the moduli spaces of Kummer surfaces associated to 
$(1,t)$-polarised  abelian surfaces.
The Kodaira dimension of the moduli space 
${\cal A}_{21}$  of $(1,21)$-polarised abelian surfaces is non-negative
because the geometric genus $h^{3,0}(\bar{\cA}_{21})$ is positive
for any smooth compactification of $\cA_{21}$ (see \cite{G1}--\cite{G-Ir}).
We have a $(4:1)$ covering 
$
{\cal A}_{21}\ \to \  {\cal A}_{21}^*
$.

\begin{theorem}\label{thm-A21} 
The moduli space ${\cal A}_{21}^*$ of  Kummer surfaces associated to
 $(1,21)$-polarised abelian surfaces is uniruled.
\end{theorem}
\begin{proof} The symplectic group of genus $2$ can be considered 
as an orthogonal group of signature $(2,3)$ (see \cite{G1} and \cite{GH-M}).
In particular, one has 
$$
\Gamma_t/\{\pm E_4\}\cong \Tilde{\SO}^+(L_t),\qquad
\Gamma_t^*/\{\pm E_4\}\cong {\Orth}^+(L_t)/\{\pm E_5\}
$$
where 
$L_t=2U\oplus \latt{-2t}$, $U\cong II_{1,1}$ 
is the hyperbolic plane (i.e. the even unimodular lattice of signature $(1,1)$),
$2U=U\oplus U$, 
$\latt{-2t}$ is the lattice of rank $1$ generated by an element of degree $-2t$, 
 $\sign(L_t)=(2,3)$ and 
$$ 
S_t=\begin{pmatrix} 
0&0&\ 0&0&1\\
0&0&\ 0&1&0\\
0&0&-2t&0&0\\
0&1&\ 0&0&0\\
1&0&\ 0&0&0
\end{pmatrix}
$$ 
is the Gram matrix of the quadratic form on $L_{t}$
in the standard basis used in \cite{GH-M}.
Therefore the moduli spaces of $(1,t)$-polarised abelian surfaces and associated 
Kummer surfaces are modular varieties of orthogonal type
$$
\cA_{t}\cong \Gamma_t\setminus \HH_2\cong
{\Tilde\SO}^+(L_t)\setminus \cD(L_t),\qquad 
\cA^*_{t}\cong \Gamma_t^*\setminus \HH_2\cong
{\Orth}^+(L_t)\setminus \cD(L_t).
$$
In what follows we construct a reflective modular form with respect 
to ${\Orth}^+(L_{21})$ and apply Theorem \ref{Th-uni}.  

It was proved in \cite[Main Theorem 2.2.3]{GN-Cl}, that 
$L_{t}$ with $t=21$ belongs to a list of special lattices
for which (meromorphic) reflective modular forms exist.
More exactly, according to this theorem
there are three (meromorphic) reflective forms for $t=21$. 
Now we construct a nearly holomorphic (reflective) Jacobi form 
$\xi_{0,21}\in J_{0,21}$ of weight $0$ and index $21$ 
whose Borcherds lifting is a holomorphic reflective modular  form. 
We define this form as a polynomial in standard  Jacobi
modular forms, namely we put

\begin{multline*}
\xi_{0,21}=
\frac{E_{4,3}}{\Delta_{12}}\biggl( E_4\phi_{0,4}
\bigl(-6E_4\phi_{0,3}^2\phi_{0,4}^2 + 10E_{4,1}\phi_{0,3}^3\phi_{0,4}
+ E_{4,2}\phi_{0,4}^3 - 5E_{4,2}\phi_{0,3}^4\bigr)
\\
+E_{4,1}E_{4,2}\phi_{0,3}\bigl(\phi_{0,3}^4 - 4\phi_{0,4}^3\bigr)\biggr)
-228\phi_{0,1}^3\phi_{0,3}^2\phi_{0,4}^3
\\
+\phi_{0,1}^2\phi_{0,3}\phi_{0,4}
\bigl(958\phi_{0,4}^3 + 240\phi_{0,2}^2\phi_{0,4}^2 + 2137\phi_{0,2}\phi_{0,3}^2\phi_{0,4} + 11\phi_{0,3}^4\bigl)
\\
 +\phi_{0,1}
\biggl(
24\phi_{0,2}\phi_{0,3}^6 - 27\phi_{0,2}^2\phi_{0,3}^4\phi_{0,4} +
(-4080\phi_{0,2}^3\phi_{0,3}^2 - 6273\phi_{0,3}^4)
\phi_{0,4}^2\\
- 8826\phi_{0,2}\phi_{0,3}^2\phi_{0,4}^3 + 30\phi_{0,4}^5
\biggr)
-75\phi_{0,3}\phi_{0,2}\phi_{0,4}^4
\\ 
+
\bigl(
7668\phi_{0,3}\phi_{0,2}^3 + 24796\phi_{0,3}^3
\bigr)
\phi_{0,4}^3
 +
\bigr(
1920\phi_{0,3}\phi_{0,2}^5 + 6513\phi_{0,3}^3\phi_{0,2}^2
\bigl)
\phi_{0,4}^2
\\ 
+(24\phi_{0,3}^3\phi_{0,2}^4 + 96\phi_{0,3}^5\phi_{0,2})
\phi_{0,4}
- 24\phi_{0,3}^5\phi_{0,2}^3
- 72\phi_{0,3}^7.
\end{multline*}
In this formula we use the following notation:
$\Delta_{12}(\tau)=\eta(\tau)^{24}\in S_{12}(\SL_2(\ZZ))$
is the Ramanujan $\Delta$-function, 
$E_{4,t}$ is the  Jacobi-Eisenstein series of weight 
$4$ and index $t$ (see \cite{EZ}) and  $\phi_{0,1}$, $\phi_{0,2}$, 
$\phi_{0,3}$, $\phi_{0,4}$ are generators of the graded ring of weak Jacobi forms 
of weight  $0$ with integral coefficients (see \cite[Theorem 1.9]{G-EG})
$$
J^{weak, \ZZ}_{0,\, *}=\ZZ[\phi_{0,1}, \phi_{0,2}, \phi_{0,3}, \phi_{0,4}].
$$
The Jacobi forms $\phi_{0,1}$, $\phi_{0,2}$, $\phi_{0,3}$ are algebraically 
independent and  $4\phi_{0,4}=\phi_{0,1}\phi_{0,3}-2\phi_{0,2}^2$.
The full class of  reflective Jacobi forms in \cite{GN-Cl} was obtained using 
a recursive procedure in terms of Jacobi forms of smaller index.
Such formulae are very long to present here.
We give in this paper a formula for  $\xi_{0,21}$ in terms of the generators.
Using explicit formulae for them written in   PARI (see \cite{GN-P2})
we can easily calculate as many terms in 
the  Fourier expansion  of $\xi_{0,21}$ as we need. 
We have 
\begin{multline}
\xi_{0,21}(\tau, z)=
\mathbf {q^{-1}} + 24 +(42r^{9} + 168r^{8} + \cdots )q 
\\
+(\mathbf {3r^{14}} + 322r^{12} + \cdots )q^2 +
(420r^{15} + 4152r^{14} +\cdots )q^3
\\
+(105r^{18} + 2016r^{17} + \cdots )q^4
+ (\mathbf{2r^{21}} + 168r^{20} + \cdots )q^5 + O(q^6) 
\end{multline}
where $q=e^{2\pi i \tau}$ and $r=e^{2\pi i z}$.
The Fourier coefficients in boldface represent all Fourier coefficients
$a(n,l)q^nr^l$ in the Fourier expansion of $\xi_{0,21}(\tau, z)$
 with indices of negative hyperbolic norm $84n-l^2<0$.

The Borcherds lifting  $B_{\xi_{0,21}}(Z)$ of the Jacobi form  $\xi_{0,21}$
(see  \cite[Theorem 2.1]{GN-P2} and  \cite[\S 2.2]{GN-Cl})
is a holomorphic  modular form of weight  $12$ and trivial character
with respect to $\Tilde\SO^+(L_{21})$.  We note that a Fourier coefficient
$a(n,l)q^nr^l$ of $\xi_{0,t}$  -- in our situation we are in the case $t=21$ -- with negative hyperbolic norm
$-D=4tn-l^2<0$ determines a divisor $H_D(l)$ with multiplicity $a(n,l)$ 
of the Borcherds automorphic product $B_{\xi_{0,t}}$
associated to $\xi_{0,t}$, where 
$$
H_D(l)=\pi_t(\{
Z=\begin{pmatrix}\tau&z\\z&\omega\end{pmatrix}
\in \HH_2\,|\, n\tau+lz+t\omega=0 \})\subset\Gamma_t\setminus \HH_2.
$$
This divisor is reflective  if and only if $D=l^2-4tn$ is a common divisor
of $4t$ and $2l$
(see \cite[Lemma 2.2]{GN-Cl}).
This means that the Borcherds product $B_{\xi_{0,21}}$
has three reflective divisors  $H_{84}(0)$, $3H_{28}(14)$ and $2H_{21}(21)$. 
Here we give an orthogonal reformulation of this fact. 
For this we represent  the index $(n,l)$ of the Fourier coefficient $a(n,l)$
as vector $(k, \frac l{2t}, 1)$ in  the dual hyperbolic lattice
$(L^{(1)}_{t})^\vee$ of $L^{(1)}_{t}=U\oplus\latt{-2t}$ (see \cite[\S 2.2]{GN-Cl}).
In the homogeneous domain $\cD(L_{{21}})$
the  Fourier coefficient $q^{-1}$ determines the reflective divisors 
$$
{\cal D}_r, \qquad r\in L_{21},\quad r^2=-2,\quad \div(r)=1,\quad {\rm mult}(D_r)=1
$$
where $\div(r)$ is the positive generator of the integral ideal 
$(r, L_{21})\subset \ZZ$.
The two other Fourier coefficients determine  the divisors
$$
\cD_u, \qquad u\in L_{21},\quad u^2=-2,\quad \div(u)=2,\quad 
 {\rm mult}(\cD_u)=2
\qquad (2q^5r^{21}),
$$
$$
\cD_v, \qquad v\in L_{21},\quad  v^2=-6,\quad  \div(v)=3,\quad {\rm mult}(\cD_v)=3 
\qquad (3q^2r^{14}).
$$
The  divisors $\cD_r$ (respectively $\cD_u$ and $\cD_v$) 
form one orbit with respect  to $\Tilde{\SO}^+(L_{21})$.
$\Orth(L_{21}^\vee/L_{21})$ is the $2$-abelian group of order $4$
(see \cite{GH-M}).
The reflection $\sigma_v$ induces a non-trivial involution
in the finite orthogonal group $\Orth(L_{21}^\vee/L_{21})$ which is different  
from $-\id$.
Therefore $\Orth^+(L_{21})=\latt{\Tilde \Orth^+(L_{21}), \sigma_v, -E_5}$.
In fact $B_{\xi_{0,21}}(Z)$ is a modular form with respect
to the full orthogonal group $\Orth^+(L_{21})$.
This is true  because the modular form $B_{\xi_{0,21}}(\sigma_v(Z))$ 
has the same divisor as $B_{\xi_{0,21}}(Z)$. Therefore
they are equal up to a constant according to the Koecher principle.

Now we can apply Theorem 2.1. The modular form $B_{\xi_{0,21}}$
of weight $12$ with respect to $\Gamma=\Orth^+(L_{21})$
has three reflective divisors 
$$
\div_{\Gamma\setminus \cD(L_{21})} B_{\xi_{0,21}}=
\pi_{\Gamma}(D_r)+2\pi_{\Gamma}(D_u)+
3\pi_{\Gamma}(D_v).
$$
Since the weight $12 > 3\cdot3$ the modular variety
$
\Orth^+(L_{21})\setminus \cD(L_{21})\cong \Gamma_{21}^*\setminus \HH_2
$ 
is uniruled according to Theorem 2.1.
\end{proof}
\medskip

\noindent
{\bf 3.2 Uniruled moduli spaces of lattice polarised $\Kthree$ surfaces.}
\smallskip

Let $S$ be a positive definite lattice. We put
$$
L(S)=2U\oplus S(-1), \qquad \sign(L(S))=(2,2+\rank S)=(2,2+n)
$$
where $S(-1)$ denotes
the corresponding negative definite lattice of rank $n$.
In the applications of this paper $S$ will be $A_n$ ($n\le 7$), $D_n$ ($n\le 8$)
and $E_6$ or direct sums of some of them.
In what follows we denote by $kL$ the orthogonal sum of $k$ copies
of the lattice $L$ and by  $L(m)$ the lattice $L$ with  quadratic form
multiplied by $m$. 

If there exists a primitive embedding of $L(S)$
into the so-called $\Kthree$ lattice  $L_{\Kthree}=3U\oplus 2E_8(-1)$ then 
the modular variety
$\Tilde\Orth^+(L(S))\setminus \cD(L(S))$
is the moduli space of lattice polarised
$\Kthree$-surfaces with transcendental   lattice $T=L(S)$.
The Picard lattice  $\Pic(X)$ of a generic member $X$ of this moduli space
is the hyperbolic lattice $L(S)^\perp_{L_{\Kthree}}$.
See \cite{N}, \cite{Do} for more details.
If $L(S)$ is $2$-elementary, i.e. $L(S)^\vee/L(S)$ is a $2$-elementary 
abelian group, then many  moduli spaces of lattice polarised $\Kthree$ surfaces
are unirational or rational (see \cite{Ma1}, \cite{Ma2}). 
Here we mainly  consider
more complicated discriminant groups. 
For $\cA_{21}^*$ the discriminant group is the cyclic group $C_{42}$.
In the  examples of this subsection 
the discriminant group is equal to $C_m$ 
($3 \le m \le 8$) and to  $C_4^2$, $C_3^3$, $C_3^2$. 
 
\begin{theorem}\label{thm-AD}
The modular variety 
$\cM(S)=\Tilde{\Orth}^+(L(S))\setminus \cD(L(S))$ of dimension 
$2+\rank S$ is uniruled
for $S$ equal to  $A_n$ $(2\le n\le 7)$, $2A_3$, $3A_2$, $2A_2$,
$D_5$, $D_7$ and $E_6$. 
\end{theorem}

We construct strongly reflective modular forms
for all $L(S)$ in the theorem using the quasi pullback of the Borcherds
form $\Phi_{12}$ (see \S 1).
We refer the reader to \cite{BKPS}, \cite{GHS1}--\cite{GHS3} for details of 
the construction of quasi pullback. The proof of the  following result can be found 
in  \cite[Theorem 8.2 and Corollary 8.12 ]{GHS3}.

\begin{theorem}\label{thm-qpb}
Let $L\emb II_{2,26}$ be a primitive nondegenerate sublattice of
signature $(2,n)$, $n\ge 3$ and $\cD_L\emb\cD_{II_{2,26}}$ be the
corresponding embedding of the homogeneous domains.  The set of
$-2$-roots
\begin{equation*}
R_{-2}(L^\perp)=\{r\in II_{2,26}\mid r^2=-2,\ (r, L)=0\}
\end{equation*}
in the orthogonal complement is finite.
We put $N(L^\perp)=\# R_{-2}(L^\perp)/2$. Then the function
\begin{equation*}\label{qpb}
  \left. \Phi|_L=
    \frac{\Phi_{12}(Z)}{
      \prod_{r\in R_{-2}(L^\perp)/{\pm 1}} (Z, r)}
    \ \right\vert_{\cD_L}
  \in M_{12+N(L^\perp)}(\Tilde{\Orth}^+(L),\, \det),
\end{equation*}
where in the product over $r$ we fix a  system of
representatives in $R_{-2}(L^\perp)/{\pm 1}$.  The modular form
$\Phi|_L$ vanishes only on rational quadratic divisors of type
$\cD_v(L)$ where $v\in L^\vee$ is the orthogonal projection to $L^\vee$
of a $-2$-root $r\in II_{2,26}$. 
Moreover $\Phi|_L$ is a cusp form if $R_{-2}(L^\perp)$ is not empty. 
\end{theorem}

To apply Theorem \ref{thm-qpb} we need basic properties of the  root
lattices 
\begin{align*}
D_n&=\{(x_1,\dots,x_n)\in \ZZ^n\,|\, x_1+\dots+x_n\in 2\ZZ\},\\
A_n&=\{(x_1,\dots,x_{n+1})\in \ZZ^{n+1}\,|\, x_1+\dots+x_{n+1}=0\}.
\end{align*}
It is known(see \cite[Ch. 4]{CS})
 that  $A_n^\vee/A_n$ is the  cyclic group of order $n+1$ and $D_n^\vee/D_n$
is isomorphic to the cyclic group of order $4$ for odd $n$ and to
$C_2\times C_2$  for even $n$. 
The discriminant forms are generated by the following elements having {\it the minimal 
possible norm} in the corresponding classes modulo $A_n$ or $D_n$:
$$
D_n^\vee/D_n=
\{\,0,\  e_n,\  (e_1+\dots+e_n)/2,\   (e_1+\dots+e_{n-1}-e_n)/2 \mod D_n\},
$$
$$
A_n^\vee/A_n=\{\, 
\varepsilon_i=\frac{1}{n+1}
(\ \underbrace{\,i,\dots, i,\,}_{\text{$n+1-i$}}\  
\underbrace{\,i-n-1,\dots,i-n-1\,}_{\text{$i$}}\ ),\ 1\le i\le n+1\,\}.
$$
If $n\le 7$ then for any $i$ we have 
$\frac{n}{n+1}
\le (\varepsilon_i, \varepsilon_i)=\frac{i(n+1-i)}{n+1}\le 2$.  
These representations of the discriminant groups of $A_n$ and $D_n$ 
show  that for all $A$- and $D$-lattices mentioned in Theorem  \ref{thm-AD}  we have 
\begin{equation}\label{eq-norm}
\forall\, \bar a\in (S^\vee/S)\ \exists\, \alpha\in \bar a :  (\alpha, \alpha)\le 2.
\end{equation}
The same property is true for $E_6$, $E_7$ and $E_8$.
The discriminant group of $E_6$ is the cyclic group of order $3$.
Each of the two non-zero classes of $E_6^\vee/E_6$ contains a vector
of square $4/3$ (see \cite[Ch. 4, \S 8.3]{CS}).

We first construct a strongly reflective modular form with respect to 
$\Tilde\Orth^+(2U\oplus A_7(-1))$. The even unimodular lattice $II_{2,26}$
is unique up to isomorphism, but it has $24$ different models
$II_{2,26}\cong 2U\oplus N(R)(-1)$ where $N(R)(-1)$ is a negative definite
even unimodular Niemeier lattice with root system $R$.
If $R$ is empty then $N(\emptyset)$ is  the Leech lattice. 
For example for $A_7$ we can take $N(2A_7\oplus 2D_5)$.
This gives us an embedding of 
$L(A_7)=2U\oplus A_7(-1)$ in $II_{2,26}\cong 2U\oplus N(2A_7\oplus 2D_5)(-1)$
and a cusp form 
$$
\Phi|_{L(A_7)}\in S_{60}(\Tilde\Orth^+(L(A_7)), \det),
\qquad (|R_2(A_7\oplus 2D_5)|=96).
$$
$\Phi|_{L(A_7)}$ is  strongly reflective.
More pecisely we shall prove that the divisor of $\Phi|_{L(A_7)}$ is  similar
to the divisor of $\Phi_{12}$, namely
\begin{equation}\label{A7}
\div (\Phi|_{L(A_7)})=
\bigcup_{\substack{r\in L(A_7)/\pm 1\vspace{1\jot}\\
(r,r)=-2}} \cD_r(L(A_7)).
\end{equation}
According to Theorem \ref{thm-qpb} the divisors of $\Phi|_{L(A_7)}$
are the rational quadratic divisors $\cD_v$ where  for $v\in L(A_7)^\vee$ 
there exists $u$ in the dual lattice of the  orthogonal complement of $A_7(-1)$ 
in the Niemeier lattice $N(2A_7\oplus 2D_5)(-1)$ such that $v+u\in II_{2,26}$ and $v^2+u^2=-2$.
If $v^2=-2$ then $u=0$ and $v$ is a $-2$-root of both lattices 
$II_{2,26}$ and $L(A_7)$. Therefore $\Phi|_{L(A_7)}$ vanishes 
along this divisor.  
We assume that $-2<v^2<0$.
According to (\ref{eq-norm}) there exists $h\in A_7^\vee(-1)$ such 
that $v\in  h+L(A_7)$ and $v^2=h^2$.  Moreover 
$v$ and $h$ are primitive in $L(A_7)^\vee$.
If not, then $\frac h{m}\in A_7^\vee(-1)$ and 
$(\frac h{m}, \frac h{m})\ge -(\frac{2}m)^2\ge -\frac {1}2$. But 
$(l,l)\le -\frac{7}8$ for any $l\in A_7^\vee(-1)$.
According to the Eichler criterion (see \cite[Proposition 3.3]{GHS4})
there exists $\gamma\in \Tilde\SO^+(L(A_7))$ such that 
$\gamma(v)=h$. Therefore $\gamma(\cD_v)=\cD_h$.
It means that one can complete  $h\in A_7^\vee(-1)$ to a root in the Niemeier lattice 
$N(2A_7\oplus 2D_5)$. This is not possible because all roots
of any Niemeier lattice $N(R)$ are roots of the root lattice $R$.
Thus property (\ref{A7}) is proved. According to Theorem 2.1
the modular variety  $\cM(A_7)$ is uniruled. The same proof  works for all the other lattices from 
Theorem 3.1. The reason is that we only used the metric property
(\ref{eq-norm}). This finishes the proof of Theorem \ref{thm-AD}.
\smallskip

\noindent
{\bf Remark 1.} One can formalize the quasi pullback consideration and to 
construct more reflective modular forms using $\Phi_{12}$ and other reflective 
modular forms. See the forthcoming paper \cite{GG}. 
\smallskip

\noindent
{\bf Remark 2.}
{We expect that one can find a similar  construction 
for the reflective modular form $B_{\xi_{0,21}}$ using 
a vector of norm $12$ in the Leech lattice and    
the pullback of the Borcherds modular form $\Phi_{12}$
(see Remark 4.4 in \cite{GN-P2}).
\medskip

\noindent
{\bf 3.3 Modular forms with the simplest  possible divisor and uniruled modular
varieties.} 
\smallskip

The divisors of the modular forms used in \S 3.2 are generated by $-2$-reflections
(see (\ref{A7})). 
It might happen that a modular group does not contain $-2$-reflections
but it contains  $-4$- or $-6$-reflections.
These divisors are simpler in the sense of \cite{G-R} because 
the Mumford-Hirzebruch volume of such modular divisors is smaller.
Three series of  strongly reflective modular forms with the simplest divisor
were constructed in \cite{G-R}.
The longest series  is the modular tower $D_1$ -- $D_8$. 
According to \cite[Theorem 3.2]{G-R}  the modular form 
$$
\Delta_{12-m,D_m}=
{\rm Lift}\,(\eta^{24-3m}(\tau)\vartheta(\tau,z_1)
\cdots \vartheta(\tau,z_m))
\in M_{12-m}(\Tilde\SO^+(L(D_m)))
$$
with  $1\le m \le 8$ has the following divisor 
\begin{equation}\label{eq-divD}
\div \Delta_{12-m,D_m} =
\bigcup_{\substack {v\in L(D_m)/ \pm 1\vspace{0.5\jot} \\  v^2=-4,\ \div(v)=2}}
\cD_v(L(D_m))
\end{equation}
with some modification for $m=4$.
In this section  we use the modular  forms with the simplest divisor for 
 $S=D_2\cong A_1\oplus A_1$, $A_2$ and  $D_3$. 
\begin{theorem}\label{thm-D3A2}
The following modular varieties are uniruled
$$
\cS\cM(D_3)=\Tilde\SO^+(L(D_3))\setminus\cD(L(D_3)),
$$
$$
\cM(2U(3)\oplus A_2(-1))=\Tilde\Orth^+(2U(3)\oplus A_2(-1))
\setminus \cD(2U(3)\oplus A_2(-1)),
$$
$$
\cS\cM^+(L(2A_1))=\Gamma \setminus \cD(L(2A_1)),
$$
where 
$\Gamma=\latt{\Tilde\SO^+(L(2A_1)), \sigma_{-4}}$ and $\sigma_{-4}$ is a reflection
acting non trivially on the discriminant group of $L(2A_1)$.
\end{theorem}
\begin{proof}
We note  that $D_3\cong A_3$.
Therefore $\cS\cM(D_3)$ 
is a double covering of the variety
$\cM(A_3)$ considered in Theorem \ref{thm-AD}.
The second variety is a covering of $\cM(A_2)$
because  $\Tilde\Orth^+(2U(3)\oplus A_2(-1))$
is a congruence subgroup of $\Tilde\Orth^+(2U\oplus A_2(-1))$.
Therefore the claim of the theorem is stronger than 
similar results of  Theorem \ref{thm-AD} for $A_2$ and $D_3$. 

In the case of  $D_3$  the ramification divisor of the modular  projection
$$
\pi_{D_3}^+:\cD(L(D_3))\to  \Tilde\SO^+(L(D_3))\setminus \cD(L(D_3))
$$
is equal to the divisor of $\Delta_{9,D_3}$
(see \cite[Lemma 2.1]{G-R}). 
Thus the first modular variety listed in the theorem
is uniruled according to Theorem 2.1.

Consider the last case of the theorem.
We note that  
$$
D_2=\latt{e_1+e_2, e_1-e_2}\cong 2A_1,\qquad 
\Orth(D_2^\vee/D_2)\cong C_2.
$$
The  only non trivial element
of $\Orth(D_2^\vee/D_2)$ is realized by the reflection $\sigma_{2e_1}$. 
All $-4$-vectors with divisor  $2$ form one $\Tilde\SO^+(L(D_2))$-orbit according 
to the Eichler criterion. 
Hence by (\ref{eq-divD}) the form
$\Delta_{10,D_2}$ is  strongly reflective  
with respect to $\Gamma$.

To prove uni\-ruled\-ness of the second modular variety
we take the modular form 
$$
\Delta_{9,A_2}={\rm Lift}(\eta^{15}(\tau)\vartheta(\tau, z_1)
\vartheta(\tau, z_2)\vartheta(\tau, z_2-z_1))
\in  S_9(\Tilde\Orth^+(L(A_2))).
$$
According to \cite[Theorem 4.2]{G-R} we have 
$$
\div \Delta_{9,A_2} =
\bigcup_{\substack {v\in L(A_2)/\pm 1\vspace{0.5\jot} \\  v^2=-6,\ \div(v)=3}}
\cD_v(L(A_2)).
$$
The modular form $\Delta_{9,A_2}$ is anti-invariant with respect 
to reflections $\sigma_v$ with $v^2=-6$, $\div(v)=3$.
These reflections induce the  non-trivial element of the finite
orthogonal discriminant group $\Orth(A_2^\vee/A_2)\cong C_2$. Therefore 
we have  $\Delta_{9,A_2}\in  S_9(\Orth^+(L(A_2)), \chi_2)$ where 
$\chi_2$ is a character of order $2$.

If $v^2=-6$ and $\div(v)=3$ then $v/3$ is a primitive vector of $L(A_2)^\vee$
and $(\frac{v}3, \frac{v}3)=-\frac{2}3$. Therefore 
$v/3$ is a $-2$ vector of the lattice 
$$
L(A_2)^\vee(3)\cong 2U(3)\oplus A_2^\vee(-3)\cong 2U(3)\oplus A_2(-1).
$$
We have $\Orth^+(L)=\Orth^+(L^\vee)=\Orth^+(L^\vee(m))$ and we can thus consider
$\Delta_{9,A_2}$ as a modular form with respect to the last group.
Since $\sigma_v=\sigma_{(v/3)}\in \Tilde\Orth^+(2U(3)\oplus A_2(-1))$
the modular form $\Delta_{9,A_2}$ is strongly reflective with respect
to  $\Tilde\Orth^+(2U(3)\oplus A_2(-1))$.
Therefore the second variety of the theorem is also uniruled.
\end{proof}

\noindent
{\bf Remark 1.} S.~Ma has informed us that he can in fact prove 
that the variety $\cS\cM^+(L(2A_1))$ is rational.
\smallskip

\noindent
{\bf Remark 2.} A modular form of type $\Phi|_{L(S)}$ where $S$ is a root lattice
from Theorem \ref{thm-AD} is the automorphic  discriminant of the moduli space 
of lattice polarised 
$\Kthree$ surfaces. It determines a Lorentzian Kac--Moody algebra
and gives an arithmetic version of mirror symmetry for $\Kthree$ surfaces
(see \cite{GN-MS} for more details). The strongly reflective modular forms
from \S 3.3 have similar interpretation for corresponding moduli spaces.
\medskip

{\noindent {\bfseries Acknowledgements}:
We should like to thank Claire Voisin for stimulating questions
during the talk of the first author in Luminy (2011).
We are grateful for financial support
under grants DFG Hu 337/6-2 and ANR-09-BLAN-0104-01.
This work is also thematically related to DFG Research Training Group 1463.
The authors would like to thank Max-Planck-Institut
f\"ur Mathematik in Bonn for support and for providing excellent
working conditions.

\end{document}